\documentclass[10pt]{article}
\font\smallit=cmti10

\usepackage{amssymb,latexsym,amsmath,epsfig,amsthm} %% Add other packages as necessary

\makeatletter

\renewcommand\section{\@startsection {section}{1}{\z@}
{-30pt \@plus -1ex \@minus -.2ex}
{2.3ex \@plus.2ex}
{\normalfont\normalsize\bfseries\boldmath}}

\renewcommand\subsection{\@startsection{subsection}{2}{\z@}
{-3.25ex\@plus -1ex \@minus -.2ex}
{1.5ex \@plus .2ex}
{\normalfont\normalsize\bfseries\boldmath}}

\renewcommand{\@seccntformat}[1]{\csname the#1\endcsname. }

\makeatother

\newtheorem{lemma}{Lemma}

\theoremstyle{definition}

\newcommand{\Z}{\mathbb{Z}}

\newtheorem{algorithm}{Algorithm}[section]

\begin{document}

\begin{center}
\uppercase{\bf The South Caicos factoring algorithm}
\vskip 20pt
{\bf Michael O. Rubinstein
\footnote{Support for work on this paper was provided by an
NSERC Discovery Grant}}\\
{\smallit Pure Mathematics, University of Waterloo, Waterloo, Ontario,  Canada}\\
\vskip 10pt
\end{center}
%\vskip 20pt
%\centerline{\smallit Received: , Revised: , Accepted: , Published: } % We will fill in the dates
%\vskip 30pt

\centerline{\bf Abstract}

\noindent
Let $N=UV$, where $U,V$ are integers, with $1< U,V <N$, and $\gcd(U,V)=1$.
We describe a probabilistic algorithm for factoring $N$ using $O(\max(U,V)^{1/2+\epsilon})$
bit operations.

\pagestyle{myheadings} 
%\markright{\smalltt INTEGERS: 20 (2020)\hfill} 
\thispagestyle{empty} 
\baselineskip=12.875pt 
\vskip 30pt

%\subjclass[2000]{11Y05; 11L050}
%\keywords{factoring, number theory, Kloosterman sums}

%\maketitle

\section{Preliminaries}
\label{sec:caicos}

Let $N=UV$, where $U,V$ are integers, with $1< U,V <N$, and $\gcd(U,V)=1$.

Let $a$ be an integer, $1<a<N$. By the division algorithm, write
\begin{eqnarray}
    \label{eq:U V}
    U=u_1 a + u_0, \quad \text{with $0 < u_0 < a$} \notag \\
    V=v_1 a + v_0, \quad \text{with $0 < v_0 < a$}.
\end{eqnarray}
If, for a given $a$, we can determine $u_0,u_1,v_0,v_1$ then we have
found $U$ and $V$. We have assumed that $u_0$ and $v_0$ are non-zero. Otherwise, $a|N$ and we
easily extract a non-trivial factor of $N$.

Previously, the author developed a factoring algorithm (called 'Hide and Seek')
requiring $O(N^{1/3+\epsilon})$ bit operations which involves studying~\eqref{eq:U V} with large $a$,
of size $N^{1/3}$. 
Details are provided in~\cite{R}.

In this paper, we describe an alternative method
for finding $u_0, v_0, u_1$ and $v_1$, 
requiring $O(\max(U,V)^{1/2+\epsilon})$ bit operations.
Thus, in the case, for example, that  both $U$ and $V$ are $O(N^{1/2})$, the algorithm
has complexity $O(N^{1/4+\epsilon})$.

Let $a$ be prime. 
We also let $a> \max(U,V)^{1/2}$, so that $u_1, v_1 < a$.
Furthermore, $u_0$ and $v_0$ are invertible modulo $a$, because $a$ is prime and $0<u_0,v_0<a$.

Our starting point is the formula
\begin{equation}
    N = (u_1 a + u_0)( v_1 a +v_0) = u_1 v_1 a^2 + ( v_0 u_1  + u_0 v_1) a + u_0 v_0
    \label{eq:main}
\end{equation}
with $ 0 < u_0, v_0 < a$, and $u_1,v_1 < a$.
Thus, subtracting $u_0 v_0$, dividing by $a$, and reducing modulo $a$, we have:
\begin{equation}
    \label{eq:start}
    (N-u_0 v_0)/a = v_0 u_1 + u_0 v_1 \mod a.
\end{equation}
We will determine $u_0, v_0, u_1, v_1$ by considering this equation.

\section{Model case}
\label{sec:model case}

We first examine the rare situation that $v_0 = u_0 \mod a$, i.e., that $a|V-U$. After explaining
the method, we will relax this assumption.

Now, from~\eqref{eq:main}, $ u_0 v_0 = N \mod a$, hence, under the assumption $v_0 = u_0 \mod a$,
\begin{equation}
    u_0^2 = N \mod a.
    \label{eq:u_0 squared}
\end{equation}
Since $a$ is assumed prime, given $N$ and $a$, we can use the Tonelli-Shanks
algorithm~\cite{S} to determine the two possible solutions to the
above equation. 

The Tonelli-Shanks algorithm requires $O(\log{a}+r^2)$ multiplications modulo $a$,
where $r$ is the power of 2 dividing $a-1$. The average value of $r$, as one averages over primes $a$,
is equal to 2 (see the appendix). Thus, on average, over primes $a$, Tonelli-Shanks requires
$O(\log{a})$ multiplications modulo $a$ to determine  the two possible values of $u_0$.
And, because we are assuming $v_0 = u_0 \mod a$, $v_0$ is determined by $u_0$.

For each of the two possible solutions $0 < u_0 < a$ to~\eqref{eq:u_0 squared},
we multiply~\eqref{eq:start} by $u_0^{-1} \mod a$.
We get, assuming $v_0 = u_0 \mod a$,
\begin{equation}
    \label{eq:step 2}
    u_0^{-1} ((N-u_0 v_0)/a) = u_1 + v_1 \mod a.
\end{equation}
But $u_1+v_1 < 2a$ (because $u_1,v_1<a$),
i.e., either $0 \leq u_1+v_1 <a$, or $a \leq u_1 + v_1 < 2a$.
Therefore, given the left-hand side of~\eqref{eq:step 2}, i.e., given
$N,a,u_0,v_0$, there are at most
two possible values for $u_1 + v_1$, which we denote by $s$.
For each of the two possible values of $s$ (and given $N,a,u_0,v_0$),
we substitute $v_1 = s - u_1$ into~\eqref{eq:main},
and solve the resulting quadratic equation in $u_1$, yielding two possible values of $u_1$,
which then also determines $v_1=s-u_1$.
We then test whether the $u_0, v_0, u_1, v_1$ thus obtained
gives a correct integer factorization of $N$.

\vspace{-.1in}
\section{Generalizing the model case}
\label{sec:general}

The model case, $v_0 = u_0 \mod a$, occurs rarely, but similar cases can be considered.
For example, say
\begin{equation}
    \beta v_0 = \alpha u_0 \mod a.
    \label{eq:alpha beta}
\end{equation}
Assume further that
\begin{eqnarray}
    \label{eq:conditions}
    &&\text{$\alpha,\beta$ are invertible modulo $a$,}\notag \\
    &&\gcd(\alpha,\beta)=1, \notag\\
    &&1 \leq \alpha \leq \beta_\text{max}/2, \notag \\
    &&-\beta_\text{max} \leq \beta \leq \beta_\text{max}/2,
\end{eqnarray}
for some positive $\beta_\text{max}$.

Equation~\eqref{eq:alpha beta} can be equivalently written as
\begin{equation}
    a | \beta V - \alpha U.
   \label{eq:beta V - alpha U}
\end{equation}
Now, $u_0 v_0 =N \mod a$, hence, by~\eqref{eq:alpha beta},
\begin{equation}
    u_0^2 = \alpha^{-1} \beta N \mod a.
    \label{eq:u_0 squared general}
\end{equation}
Thus, given $N,\alpha,\beta$, and prime $a$, we can again use the Tonelli-Shanks algorithm to determine the
two possible values of $u_0 \mod a$.

Hence, multiplying~\eqref{eq:start} by $\beta u_0^{-1} \mod a$, we get
\begin{equation}
    \label{eq:conclusion general}
    \beta u_0^{-1} ((N-u_0 v_0)/a) =  \alpha u_1  + \beta  v_1 \mod a.
\end{equation}
But, because of our assumed bounds on $\alpha$ and $\beta$, we have 
\begin{equation}
    - \beta_\text{max} a < \alpha u_1 + \beta v_1  < \beta_\text{max} a.
\end{equation}
Hence, given the left-hand side of~\eqref{eq:conclusion general},
there are at most $2 \beta_{\max}$ possibilities for 
\begin{equation}
\label{eq:s}
    s=\alpha u_1 + \beta v_1,
\end{equation}
i.e., one per interval of length $a$.

For each of the possible values of $s$ (and given $N,a,u_0,v_0,\alpha,\beta$),
we substitute $v_1 = (s - \alpha u_1)/\beta$ into~\eqref{eq:main},
and solve the resulting quadratic equation in $u_1$, yielding two possible values of $u_1$,
from which we also determine $v_1 = (s - \alpha u_1)/\beta$.
We then test whether the $u_0, v_0, u_1, v_1$ thus obtained
gives a correct integer factorization of $N=(u_1 a + u_0)(v_1 a + v_0)$.

Note that if $u_0$ leads to a positive integer factorization of $N=UV$, then the other solution $-u_0 \mod a$
to~\eqref{eq:u_0 squared general} produces the factorization $N=(-U)(-V)$.

\section{The South Caicos Algorithm}
\label{sec:algorithm}

We are now ready to describe our South Caicos factoring algorithm.

Initially, assume that $\max(U,V) < (2N)^{1/2}$.
In Section~\ref{sec:remove condition}, we will remove this assumption.

This condition holds, for example, if $U < V < 2U$ , since then $V^2 < 2UV = 2N$.
But because the method of the previous
section does not distinguish $U<V$, we prefer to state the condition as we have.

The idea is to loop through a small
number of values of $\alpha$ and $\beta$, as determined by $\beta_\text{max}=2$, say,
and primes, $(2N)^{1/4}< a < 2 (2N)^{1/4}$,
and apply the method of Section~\ref{sec:general}.
%The bound on $a$ is such that $a>V^{1/2}$, hence $u_1,v_1<a$ in ~\eqref{eq:U V}.

If, for given $(\alpha,\beta)$, we encounter
a prime $(2N)^{1/4} < a < 2 (2N)^{1/4}$ such that $a|\beta V - \alpha U$, then,
for that choice of $\alpha,\beta, a$, the method of Section~\ref{sec:general}
quickly uncovers $u_0,v_0,u_1,v_1$, and hence $U$ and $V$.

However, if, for our given set of $(\alpha,\beta)$'s, no such $(2N)^{1/4} < a < 2 (2N)^{1/4}$ is encountered,
then we can repeat the process with the same set of primes $a$, but with $\beta_\text{max}$ replaced, say,
with $\beta_\text{max}+2$, taking care to exclude $(\alpha,\beta)$'s already tested.

Heuristically, as $\beta_\text{max}$ grows,
we quickly expect to find $(\alpha,\beta)$, and a prime
$(2N)^{1/4} < a < 2 (2N)^{1/4}$, such that~\eqref{eq:beta V - alpha U} holds.
A complexity analysis follows after the pseudo code below.

\begin{algorithm}[South Caicos]
Let $N=UV$, with $U,V>1$ positive integers to be determined satisfying $\gcd(U,V)=1$, satisfying $\max(U,V) < (2N)^{1/2}$.
\begin{itemize}
    \item[1] Let $\beta_\text{max}=2$, and let $S(\text{old})$ be the empty set.
    \item[2] Let
        \begin{eqnarray*}
            S(\beta_\text{max}) &=&\{(\alpha,\beta) \in \Z^2:
            \gcd(\alpha,\beta)=1,
            \alpha \in [1 , \beta_\text{max}/2],\\
            &&\beta \in [-\beta_\text{max}, \beta_\text{max}/2], \beta \neq 0
        \}.
        \end{eqnarray*}
    \item[3] Let $a$ to be the first prime  $> (2N)^{1/4}$.
    \item[4] Use the Euclidean algorithm to compute $d=\gcd(N,a)$.
        If $d>1$ then we have determined a non-trivial factor of $N$ and quit.
    \item[5] For $(\alpha,\beta) \in S(\beta_\text{max})-S(\text{old})$:\\
    \phantom{}\hspace{3ex}
    Carry out the procedure described in Section~\ref{sec:general} for given $N,a, \alpha,\beta$.\\
    \phantom{}\hspace{3ex}
    If this results in a non-trivial integer factorization of $N$, then quit.
    %\phantom{}\hspace{3ex}
    \item[6]
    Replace $a$ by the next prime, and, if $a< 2 (2N)^{1/4}$, repeat from Step 4.
    \item[7] If $\beta_\text{max}+2 < (2N)^{1/4}$, replace $S(\text{old})$ by $S(\beta_\text{max})$, $\beta_\text{max}$ by $\beta_\text{max}+2$, and repeat from Step 2, but, henceforth, skipping over Step 4. Otherwise exit.
\end{itemize}
\end{algorithm}
Note that we do not invoke the invertibility condition of~\eqref{eq:conditions} in our definition
of $S(\beta_\text{max})$. Instead, we assume that $\beta_\text{max} < (2N)^{1/4}$, and also $\beta \neq 0$.
Because $a>(2N)^{1/4}$ is prime, this guarantees $\alpha,\beta$ are invertible $\mod a$.
We expect the algorithm to produce a factorization of $N$ well before the exit condition is reached.
See the discussion below.

Analysis:
The success and efficiency of the method hinges on encountering a prime $(2N)^{1/4} < a < 2 (2N)^{1/4}$,
and relatively small integers $\alpha,\beta$, such that $a|\beta V - \alpha U$.
Heuristically, for $U,V$ much larger than, and relatively prime to $a$, and $\gcd(U,V)=1$, we expect
$\beta V -\alpha U$ to be divisible by $a$, on average over $S(\beta_\text{max})$, $1/a$ of the time.

More precisely, letting $X=(2N)^{1/4}$, we expect, as $X\to \infty$ and $|S(\beta_\text{max})|/\log{X} \to \infty$ (but also with $\beta_\text{max} < X)$,
the number of triples $\alpha,\beta,a$, with
$a|\beta V -\alpha U$, $X < a < 2X$, and $(\alpha,\beta) \in S(\beta_\text{max})$,  to satisfy
\begin{equation}
    \label{eq:13}
    \sum_{\substack{X < a < 2 X\\ a \text{ prime}}} \sum_{\substack{(\alpha,\beta) \in S(\beta_\text{max})\\ a| \beta V-\alpha U}} 1
    \sim  |S(\beta_\text{max})|
    \sum_{\substack{X < a < 2 X\\ a \text{ prime}}}  1/a
    \sim |S(\beta_\text{max})| \log(2)/\log(X).
\end{equation}
The last step follows from the Prime Number Theorem and a summation by parts, or else using
the elementary estimate $\sum_{\substack{a < Y\\ a \text{ prime}}}  1/a \sim \log\log(Y) + b + O(1/\log(Y))$, where $b$ is a
constant, and noting that $\log\log(2X) - \log\log(X) = \log ((\log(2)+\log(X))/\log(X)) \sim \log(2)/\log(X)$.

However, from the definition of $S(\beta_\text{max})$,
\begin{equation}
    \label{eq:14}
    |S(\beta_\text{max})| \sim \frac{6}{\pi^2} \frac{3}{4} \beta_\text{max}^2,
\end{equation}
with the factor $6/\pi^2$ to account for the condition $\gcd(\alpha,\beta)=1$.
Thus, by~\eqref{eq:13} and~\eqref{eq:14}, as $\beta_\text{max}/\log(N)^{1/2}$ grows,
we expect to encounter at least one $(\alpha,\beta)
\in S(\beta_\text{max})$, and a prime $X < a < 2X$, with $X=(2N)^{1/4}$, such that $a|\beta V -\alpha U$,
and hence such that the method
of Section~\ref{sec:general} with succeed in finding non-trivial factors $U,V$ of $N$. We also note that this should
occur long before we trigger the exit condition of Step 7, since $\log(N)^{1/2}$ grows much slower than
$(2N)^{1/4}$.

The bulk of the work, per $(\alpha,\beta,a)$, involves one application of the Tonelli-Shanks algorithm
in Equation~\eqref{eq:u_0 squared general}, followed by the extraction of the roots of $2 \beta_\text{max}$
quadratic equations, one per each value of $s$ from~\eqref{eq:s}.

For each candidate $X< a < 2X$, primality testing of $a$ can be done in polynomial time.
Alternatively, one can sieve for all primes in the interval using the sieve of Eratosthenes,
at a cost of $O(X^{1/2}/\log{X})$, i.e., $O(N^{1/8}/\log{N})$ bits of storage, needed to keep track
of multiples of the primes $< (2X)^{1/2}$ as we carry out the sieve in short intervals. A table of
primes $< (2X)^{1/2}$ needed to carry out the sieve can also be tabulated using the sieve of
Eratosthenes.

Overall, we expect this algorithm to successfully factor $N$ in $O(N^{1/4+\epsilon})$ bit operations.
With this stated efficiency, the method is probabilistic,
since it relies on finding a prime $X<a<2X$, and small $\alpha, \beta$,
i.e., of order $N^\epsilon$, such that $a|\beta V -\alpha U$.

\section{Example}

For example, if $N=23713634802068266491347$,
the algorithm first uncovers the triple $a=804901$, $\alpha = 1$, $\beta=3$, with
$u_0 = 523125$, $v_0 = 174375$, being a solution to $\beta v_0 = \alpha u_0 \mod a$, and $u_0 v_0 = N \mod a$,
found by applying Tonelli-Shanks to~\eqref{eq:u_0 squared general}.
Then, following the method in Section~\ref{sec:general},
we obtain $u_1 = 235108$, $v_1 = 155684$ (with the value of $s$ that succeeds in~\eqref{eq:s}
being $s= 702160$) , giving a correct factorization of $N = UV$, with
$U=u_1 a + u_0= 189239187433$, $V=v_1 a + v_0=125310381659$.

In table~\ref{table:example} we
list additional triples $a$, $\alpha$, $\beta$, with $\beta_\text{max}=16$, such that
$a| \beta V - \alpha U$, and the corresponding
values of $u_0$, $v_0$, $s$, $u_1$, $v_1$, $U$ and $V$, produced by our method.

\begin{table}[h]
\centerline{
\begin{tabular}{|c|c|c|c|c|c|c|c|c|c|}
\hline
$a$ & $\alpha$ & $\beta$ & $u_0$ & $v_0$ & $s$ & $u_1$ & $v_1$ & $U$ & $V$ \\ \hline
804901 & 1 & 3 & 523125 & 174375 & 702160 & 235108 & 155684 & 189239187433 & 125310381659\\
%804901 & 1 & 3 & 281776 & 630526 & -702164 & -235109 & -155685 & -189239187433 & -125310381659 \\
804901 & 3 & 1 & 174375 & 523125 & 702160 & 155684 & 235108 & 125310381659 & 189239187433 \\
%804901 & 3 & 1 & 630526 & 281776 & -702164 & -155685 & -235109 & -125310381659 & -189239187433 \\
546671 & 1 & -7 & 268355 & 274047 & -2193938 & 229224 & 346166 & 125310381659 & 189239187433 \\
%546671 & 1 & -7 & 278316 & 272624 & 2193944 & -229225 & -346167 & -125310381659 & -189239187433 \\
%601291 & 4 & -5 & 318669 & 466614 & -216873 & -314722 & -208403 & -189239187433 & -125310381659 \\
601291 & 4 & -5 & 282622 & 134677 & 216874 & 314721 & 208402 & 189239187433 & 125310381659 \\
%837043 & 3 & -7 & 331050 & 814742 & 369706 & -226081 & -149707 & -189239187433 & -125310381659 \\
837043 & 3 & -7 & 505993 & 22301 & -369702 & 226080 & 149706 & 189239187433 & 125310381659 \\
601291 & 5 & -4 & 134677 & 282622 & -216874 & 208402 & 314721 & 125310381659 & 189239187433 \\
%601291 & 5 & -4 & 466614 & 318669 & 216873 & -208403 & -314722 & -125310381659 & -189239187433 \\
685099 & 6 & -7 & 456554 & 293767 & 376970 & 276221 & 182908 & 189239187433 & 125310381659 \\
%685099 & 6 & -7 & 228545 & 391332 & -376969 & -276222 & -182909 & -189239187433 & -125310381659 \\
%546671 & 7 & -1 & 272624 & 278316 & -2193944 & -346167 & -229225 & -189239187433 & -125310381659 \\
546671 & 7 & -1 & 274047 & 268355 & 2193938 & 346166 & 229224 & 189239187433 & 125310381659 \\
644153 & 1 & 7 & 77804 & 563246 & 2250988 & 194535 & 293779 & 125310381659 & 189239187433 \\
%644153 & 1 & 7 & 566349 & 80907 & -2250996 & -194536 & -293780 & -125310381659 & -189239187433 \\
%644153 & 7 & 1 & 80907 & 566349 & -2250996 & -293780 & -194536 & -189239187433 & -125310381659 \\
644153 & 7 & 1 & 563246 & 77804 & 2250988 & 293779 & 194535 & 189239187433 & 125310381659 \\
685099 & 7 & -6 & 293767 & 456554 & -376970 & 182908 & 276221 & 125310381659 & 189239187433 \\
%685099 & 7 & -6 & 391332 & 228545 & 376969 & -182909 & -276222 & -125310381659 & -189239187433 \\
837043 & 7 & -3 & 22301 & 505993 & 369702 & 149706 & 226080 & 125310381659 & 189239187433 \\
%837043 & 7 & -3 & 814742 & 331050 & -369706 & -149707 & -226081 & -125310381659 & -189239187433 \\
743161 & 7 & -16 & 60161 & 670393 & -2893914 & 168618 & 254640 & 125310381659 & 189239187433 \\
\hline
\end{tabular}
}
\caption
{We list, for $N=23713634802068266491347$ the values of prime $a$, $1\leq \alpha \leq 8$, 
$-16 \leq \beta  \leq 8$, such that the method of Section~\ref{sec:general} produces
values of $u_0$, $v_0$, $u_1$, $v_1$ that give a correct positive integer factorization
of $N$. We also list those parameters, along with the corresponding value of $s$ in~\eqref{eq:s},
and the values of $U$ and $V$.
}
\label{table:example}
\end{table}

\section{Removing the assumption $\max(U,V) < (2N)^{1/2}$}
\label{sec:remove condition}

The assumption that $\max(U,V) < (2N)^{1/2}$ was made so that, with $a>(2N)^{1/4}$,
one has, for given $a$, that $u_1,v_1 <a$. This is important in Equation~\eqref{eq:s} so that
we only need to check $2 \beta_{\max}$ possibilities for $s$.

However, we need not assume this bound on $\max(U,V)$.

Let $X=(2N)^{1/4}$.
We run the algorithm of Section~\ref{sec:algorithm}, but, at the $j$-th iteration of Step 3,
we change it to read 'let $a$ be the first prime $> 2^{j-1} X$, and in Step 6, replace '$2(2N)^{1/4}$'
with '$2^j X$'. We also use, for given $N$, the value $\beta_\text{max} = j \log{N}$, and
eliminate $S(\text{old})$.

Thus, at the $j$-th iteration, we look at sets of ever larger primes $2^{j-1} X < a < 2^j X$. For $j$ sufficiently large,
we have $a> \max(U,V)^{1/2}$, and thus $u_1, v_1 < a$, as needed for
the method of Section~\ref{sec:general} to succeed.

The large value of $\beta_\text{max}$
relative to $\log(N)^{1/2}$, and the analysis of Section~\ref{sec:general}, suggests that,
with probability tending to 1, as $N \to \infty$, that we will thus succeed in factoring $N$ using
$O(\max(U,V)^{1/2+\epsilon})$ bit operations.

\begin{algorithm}[South Caicos B]
Let $N=UV$, with $U,V>1$ positive integers to be determined satisfying $\gcd(U,V)=1$.
\begin{itemize}
    \item[1] Let $\beta_\text{max}=\log{N}$, $j=1$, and $X=(2N)^{1/4}$.
    \item[2] Let
        \begin{eqnarray*}
            S(\beta_\text{max}) &=&\{(\alpha,\beta) \in \Z^2:
            \gcd(\alpha,\beta)=1,
            \alpha \in [1 , \beta_\text{max}/2],\\
            &&\beta \in [-\beta_\text{max}, \beta_\text{max}/2], \beta \neq 0
        \}.
        \end{eqnarray*}
    \item[3] Let $a$ to be the first prime  $> 2^{j-1} X$.
    \item[4] Use the Euclidean algorithm to compute $d=\gcd(N,a)$.
        If $d>1$ then we have determined a non-trivial factor of $N$ and quit.
    \item[5] For $(\alpha,\beta) \in S(\beta_\text{max})$:\\
    \phantom{}\hspace{3ex}
    Carry out the procedure described in Section~\ref{sec:general} for given $N,a, \alpha,\beta$.\\
    \phantom{}\hspace{3ex}
    If this results in a non-trivial integer factorization of $N$, then quit.
    \item[6]
    Replace $a$ by the next prime, and, if $a< 2^j X$, repeat from Step 4.
    \item[7] Replace $j$ by $j+1$, $\beta_\text{max}$ by $j \log{N}$, and repeat from Step 2.
\end{itemize}
\end{algorithm}

\section{Appendix}

We justify the assertion made in Section~\ref{sec:model case} regarding the average value of $r$ that
appears in the Tonelli-Shanks algorithm.

\begin{lemma}
Let $a$ be prime, and $r$ the power of 2 dividing $a-1$.
Then, the average value of $r$ tends to 2, when averaged over primes $A < a \leq 2A$, as $A \to \infty$.
\end{lemma}

\begin{proof}
Let $k$ be a positive integer.
If $a=m \mod 2^k$, with $m$ odd and $1 \leq m < 2^k$, then the value
of $r$, the power of 2 dividing $a-1$, is equal to
\begin{eqnarray}
   &\text{1, if $m-1=2,6,10,14,\ldots$} \notag \\
   &\text{2, if $m-1=4,12,20,28,\ldots$} \notag \\
   &\text{3 if $m-1=8,24,40,56,\ldots$} \notag \\
   &\text{etc.} \notag
\end{eqnarray}
More precisely, if we write $m$ as a $k$ bit binary number (possibly with some
leading zeros), then $r=1$ if $m$ ends in $11$, $r=2$ if $m$ ends in $101$, $r=3$ if $m$ ends in $1001$, etc.
In particular, $2^{k-2}$ of these $m$ have $r=1$, $2^{k-3}$ have $r=2$, $2^{k-4}$ have $r=3$, $\ldots$, one has
$r=k-1$ (namely $m=2^{k-1}+1$).
The residue class $m=1$ requires more careful consideration.
If $m=1$, then the value of $r$ is not precisely determined, but rather satisfies, for $a<2A$,
\begin{equation}
    \label{eq: r}
    k \leq r \leq \log(2A)/\log(2).
\end{equation}

Now, the primes are
equi-distributed amongst the odd residue classes mod $2^k$. However, we require slightly more than just
the main term of the Prime Number Theorem in arithmetic progressions. Specifically, let
$c>0$, and $q$ a positive integer with $q \leq \log(x)^c$.
The Siegel-Walfisz Theorem implies that, if $\gcd(m,q)=1$ then, $\pi(x;q,m)$,
the number of primes less than or equal to $x$ and congruent to $m \mod q$, satisfies
\begin{equation}
   \label{eq:siegel-walfisz}
   \pi(x;q,m) = \frac{1}{\phi(q)} \frac{x}{\log{x}} (1+o(1)),
\end{equation}
as $x\to \infty$, with the implied constant dependent on $c$, and ineffective. If we assume the GRH, then this
holds with the implied constant effectively computable (and also a much stronger remainder term).
%in S-W, rhs is stated as Li(x)/phi(q) + O(x exp(-kappa log(x)^(1/2))) where kappa depends on c, and is not effective
%because Siegel's theorem is not effective. But exp(kappa log(x)^(1/2)) is much larger than log(x)^c >= q > phi(q), so
%the error term is o(1) times the main term uniformly when q <= log(x)^c.
Thus, for $k$ satisfying, say,
\begin{equation}
    \label{eq:k size}
    \log(A)^2 < 2^k \leq 2\log(A)^2,
\end{equation}
we have, unconditionally,
\begin{equation}
   \label{eq:pnt}
   \pi(2A,2^k,m) - \pi(A,2^k,m) = \frac{1}{2^{k-1}} \frac{A}{\log{A}}(1+o(1)),
\end{equation}
as $A\to \infty$.

Counting the contribution from each residue class $m \mod 2^k$, and taking into account~\eqref{eq: r} and~\eqref{eq:pnt},
the average value of $r$, over primes $A < a \leq 2A$, is equal to:
\begin{equation}
    \label{eq:average} 
    \frac{1}{\pi(2A)-\pi(A)} \left(\sum_{r=1}^{k-1} r 2^{k-r-1} + O(\log{A})\right) \ \frac{1}{2^{k-1}} \frac{A}{\log{A}}(1+o(1)).
\end{equation}
But the sum in parentheses is equal to $2^k-k-1$, as can be verified inductively. Furthermore,
$\pi(2A)-\pi(A) \sim A/\log{A}$. Thus, the above equals
\begin{equation}
    \label{eq:average b} 
    \left(2 + O( (\log{A}+k)/2^k )\right)(1+o(1)).
\end{equation}
But, by~\eqref{eq:k size}, $(\log(A)+k)/2^k \to 0$ as $A \to \infty$. Hence, the average value of $r$
tends to 2 as $A \to \infty$.
\end{proof}

\noindent We note that condition~\eqref{eq:k size} is used in two places. We need $2^k$ to grow faster than $\log(A)$
so as to get the limiting value of 2 in Equation~\eqref{eq:average b}. We also invoke the Siegel-Walfisz
theorem in~\eqref{eq:siegel-walfisz} which gives a uniform estimate for the Prime Number Theorem in arithmetic
progressions, so long as the modulus $2^k$ grows slower than a power of $\log(A)$, hence the assumption
that $2^k < 2 \log(A)^2$.

\noindent {\bf Acknowledgement.}
The above algorithm was developed by the author in South Caicos while on vacation
with his lovely girlfriend Lisa, in between snorkeling, drinking, and getting chased by rabid dogs on the beach.

\end{document}